\documentclass[11pt]{amsart}
\usepackage[margin=35mm]{geometry}
\usepackage{amsmath,amssymb}
\usepackage{amsthm}

\newtheorem{thm}{Theorem}[section]
\newtheorem{prop}{Proposition}[section]
\newtheorem{lem}{Lemma}[section]
\newtheorem{cor}{Corollary}[section]
\newtheorem{defi}{Definition}[section]

\newtheorem{rem}{Remark}[section]

\begin{document}

\title{Distributionally chaotic maps are $C^0$-dense}
\author{Noriaki Kawaguchi}
\address{Graduate School of Mathematical Sciences, The University of Tokyo, 3-8-1 Komaba Meguro, Tokyo 153-8914, Japan}
\email{knoriaki@ms.u-tokyo.ac.jp}

\maketitle

\markboth{NORIAKI KAWAGUCHI}{Distributionally chaotic maps are $C^0$-dense}

\begin{abstract}
We prove that the set of maps which exhibit distributional chaos of type 1 (DC1) is $C^0$-dense in the space of continuous self-maps of given any compact topological manifold (possibly with boundary). 
\end{abstract}

\section{Introduction}

Throughout this paper, $X$ denotes a compact metric space endowed with a metric $d$. We denote by $C(X)$ the space of continuous self-maps of $X$, equipped with the metric $d_{C^0}$ defined by \[d_{C^0}(f,g)=\sup_{x\in X}d(f(x),g(x))\] for $f,g\in C(X)$. The aim of this paper is to prove the following theorem.

\begin{thm}
Let $M$ be a compact topological manifold (possibly with boundary). Then, the set of maps which exhibit distributional chaos of type 1 (DC1) is dense in $C(M)$.
\end{thm}

The notion of {\em distributional chaos} was introduced by Schweizer and Sm\'ital \cite{SS}. For a continuous map $f:X\to X$, a pair of points $(x,y)\in X^2$ is said to be a {\em DC1-pair} for $f$ if
\begin{equation*}
\limsup_{n\to\infty}\frac{1}{n}|\{0\le i\le n-1:d(f^i(x),f^i(y))<\delta\}|=1\quad\text{for every $\delta>0$},
\end{equation*}
and
\begin{equation*}
\limsup_{n\to\infty}\frac{1}{n}|\{0\le i\le n-1:d(f^i(x),f^i(y))>\delta_0\}|=1\quad\text{for some $\delta_0>0$}.
\end{equation*}
Then, $f$ is said to exhibit {\em distributional chaos of type 1} (DC1) if there is an uncountable set $S\subset X$ such that  for any $x,y\in S$ with $x\ne y$, $(x,y)$ is a DC1-pair for $f$. Indeed, the distributional chaos has two more versions: DC2 and DC3. They are numbered in the order of decreasing strength, so DC1 is the strongest, and DC2 is stronger than Li-Yorke chaos.

Hereafter, $M$ denotes a compact topological manifold (possibly with boundary). For any $f\in C(X)$, we denote by $h_{top}(f)$ the topological entropy of $f$. In \cite{Y}, it was proved by Yano that a generic $f\in C(M)$ satisfies $h_{top}(f)=\infty$. Downarowicz proved that for any $f\in C(X)$, $h_{top}(f)>0$ implies DC2 \cite{D}. These together imply that DC2 is generic in $C(M)$. Then, it is natural to ask whether DC1 is still generic in $C(M)$ or not. When $M$ is the unit interval (denoted by $I$), we know that for any $f\in C(I)$, $h_{top}(f)>0$ iff $f$ exhibits DC1 \cite{SS}; therefore, DC1 is generic in $C(I)$. Theorem 1.1 gives a partial answer to the question for a general $M$. It should be noticed that Piku\l a showed that $h_{top}(f)>0$ does not necessarily imply DC1 for any $f\in C(X)$ \cite{Pik}. This indicates that some additional assumptions besides positive topological entropy are needed to ensure DC1 for general continuous maps.

{\em Shadowing} is a natural candidate for such an assumption. Recently in \cite{LLT}, Li {\it et al.} proved that for any $f\in C(X)$ with the shadowing property, $f$ exhibits DC1 if one of the following properties holds: (1) $f$ is non-periodic transitive and has a periodic point, or (2) $f$ is non-trivial weakly mixing. Here, note that we have $h_{top}(f)>0$ in both cases. According to Mazur and Oprocha \cite{MO}, we know that the shadowing property is generic in $C(M)$. However, because of the additional assumption (1) or (2), it is not obvious that these results imply the genericity of DC1.

The {\em limit shadowing property} introduced by Eirola {\it et al.} in \cite{ENP} is a variant of the shadowing property defined as follows. Given a continuous map $f:X\to X$, a sequence $(x_i)_{i\ge0}$ of points in $X$ is a {\em limit pseudo orbit} of $f$ if $\lim_{i\to\infty}d(f(x_i),x_{i+1})=0$. Then, $f$ is said to have the {\em limit shadowing property} if for any limit pseudo orbit  $(x_i)_{i\ge0}$ of $f$, there is $x\in X$ such that $\lim_{i\to\infty}d(f^i(x),x_i)=0$, where such $x$ is called a {\em limit shadowing point} of $(x_i)_{i\ge0}$. The property provides a method of gaining information about the asymptotic behavior of the true orbits from pseudo-orbits of continuous maps. Then, we prove the following theorem.

\begin{thm}
Suppose that a continuous map $f:X\to X$ satisfies the following properties. 
\begin{itemize}
\item[(1)] The limit shadowing property.
\item[(2)] $h_{top}(f)>0$.
\end{itemize}
Then, $f$ exhibits distributional chaos of type 1 (DC1). Moreover, $E(X,f)\subset\overline{DC1(X,f)}$, where $E(X,f)$ (resp. $DC1(X,f)$) is the set of entropy pairs (resp. DC1-pairs) for $f$.
\end{thm}

This theorem guarantees DC1 for any $f\in C(X)$ with the limit shadowing property without any other assumption than the positive topological entropy. In \cite{MO}, it was also proved that the set of maps with the {\em s-limit shadowing property} is dense in $C(M)$. We know that the s-limit shadowing property always implies the limit shadowing property \cite{BGO}. On the other hand, Yano proved that for any $K>0$, the set \[E_K(M)=\{f\in C(M):h_{top}(f)\ge K\}\] contains an open and dense subset of $C(M)$ \cite{Y}. Thus, Theorem 1.1 is implied by Theorem 1.2.

This paper consists of four sections. Some basic definitions and facts are briefly collected in Section 2. In Section 3, we introduce a relation on the chain recurrent set, and prove a few lemmas needed for the proof of Theorem 1.2. Then, Theorem 1.2 is proved in Section 4.

\section{Preliminaries}

In this section, we collect some basic definitions and facts used in this paper.

\subsection{{\it Chains, cycles, pseudo-orbits, and the shadowing property}}

Given a continuous map $f:X\to X$, a finite sequence $(x_i)_{i=0}^{k}$ of points in $X$, where $k$ is a positive integer, is called a {\em $\delta$-chain} of $f$ if $d(f(x_i),x_{i+1})\le\delta$ for every $0\le i\le k-1$. A $\delta$-chain $(x_i)_{i=0}^{k}$ of $f$ is said to be a {\em $\delta$-cycle} of $f$ if $x_0=x_k$. For $\delta>0$, a sequence $(x_i)_{i\ge0}$ of points in $X$ is called a {\em $\delta$-pseudo orbit} of $f$ if $d(f(x_i),x_{i+1})\le\delta$ for all $i\ge0$. Then, for $\epsilon>0$, a $\delta$-pseudo orbit  $(x_i)_{i\ge0}$ of $f$ is said to be {\em $\epsilon$-shadowed} by $x\in X$ if $d(f^i(x),x_i)\leq \epsilon$ for all $i\ge 0$. We say that $f$ has the {\em shadowing property} if for any $\epsilon>0$, there is $\delta>0$ such that every $\delta$-pseudo orbit of $f$ is $\epsilon$-shadowed by some point of $X$. A point $x\in X$ is said to be a {\em chain recurrent point} for $f$ if for any $\delta>0$, there is a $\delta$-cycle $(x_i)_{i=0}^{k}$ of $f$ such that $x_0=x_k=x$. The set of chain recurrent points for $f$ is denoted by $CR(f)$.

\subsection{{\it Entropy pairs}}

Given a continuous map $f:X\to X$ and an open cover $\mathcal{U}$ of $X$, we denote by $h(f,\mathcal{U})$ the {\em entropy of $f$ relative to $\mathcal{U}$} (see \cite{W} for details). The notion of {\em entropy pairs} was introduced by Blanchard \cite{B}. A pair of points $(x,y)\in X^2$ with $x\ne y$ is said to be an {\em entropy pair} for $f$ if for any closed neighborhoods $A$ of $x$ and $B$ of $y$, we have $h(f,\{A^c,B^c\})>0$ whenever $A\cap B=\emptyset$. The set of entropy pairs for $f$ is denoted by $E(X,f)\:(\subset X^2)$.

For two continuous maps $f:X\to X$ and $g:Y\to Y$, we say that $(Y,g)$ is a {\em factor} of $(X,f)$ if there is a surjective continuous map $\pi:X\to Y$ such that $\pi\circ f=g\circ\pi$. Such a map is called a {\em factor map}, and also denoted as $\pi:(X,f)\to(Y,g)$. The basic properties of entropy pairs is summarized in the following lemma.

\begin{lem}\cite{B}
Given a factor map $\pi:(X,f)\to (Y,g)$, we have the following properties.
\begin{itemize}
\item[(1)] $h_{top}(f)>0$ if and only if $E(X,f)\ne\emptyset$.
\item[(2)] For any $x,y\in X$, if $(x,y)\in E(X,f)$ and $\pi(x)\ne\pi(y)$, then $(\pi(x),\pi(y))\in E(Y,g)$.
\item[(3)] For any $(z,w)\in E(Y,g)$, there is $(x,y)\in\pi^{-1}(z)\times\pi^{-1}(w)$ such that $(x,y)\in E(X,f)$.
\end{itemize} 
\end{lem}

We say that a continuous map $f:X\to X$ has {\em upe} if for any $(x,y)\in X^2$ with $x\ne y$, $(x,y)\in E(X,f)$. It is known that if $f$ has upe, then $f$ is weakly mixing, and when $f$ has the shadowing property, the converse holds \cite{LO}. By this, especially, the shift map $\sigma:\{0,1\}^\mathbb{N}\to\{0,1\}^\mathbb{N}$ has upe.

\section{A chain relation and a few lemmas}

In this section, we prove a few lemmas needed for the proof of Theorem 1.2. First, for any continuous map $f:X\to X$, we define a chain relation $\sim$ on $CR(f)$ as follows.

\begin{defi}
For any $x,y\in CR(f)$, $x\sim y$ if and only if for every $\delta>0$, there are integers $m=m(\delta)>0$ and $N=N(\delta)>0$ such that for any integer $n\ge N$, there are two $\delta$-chains $(x_i)_{i=0}^{mn}$, $(y_i)_{i=0}^{mn}\subset CR(f)$ of $f$ such that $x_0=y_{mn}=x$ and $x_{mn}=y_0=y$.
\end{defi}

\begin{rem}
\normalfont
Even if we replace `$\subset CR(f)$' with `$\subset X$' in the above definition, the relation $\sim$ does not change. This can be seen as follows. Given any $\epsilon>0$, there is $\delta=\delta(\epsilon)>0$ such that for any $\delta$-cycle $(z_i)_{i=0}^k$ of $f$, $d(z_i,CR(f))\le\epsilon$ holds for all $0\le i\le k$. Fix $x,y\in CR(f)$ and let $\alpha=(x_i)_{i=0}^l$, $\beta=(y_i)_{i=0}^l\subset X$ be $\delta$-chains of $f$ with $x_0=y_l=x$ and $x_l=y_0=y$. Then, since $\gamma=\alpha\beta$ is a $\delta$-cycle of $f$, the above property holds. Put $x'_0=x$, $x'_l=y$, and take $x'_i\in CR(f)$ with $d(x_i,x'_i)=d(x_i,CR(f))\le\epsilon$ for each $0<i<l$. Then, for any $\delta'>0$, if $\delta$ is sufficiently small, $\alpha'=(x'_i)_{i=0}^l\subset CR(f)$ gives a $\delta'$-chain of $f$, which has the same length and end points as $\alpha$. This argument also applies to $\beta$. We refer to the Robinson's proof of $CR(f|_{CR(f)})=CR(f)$ in \cite{Ro} for a similar argument.
\end{rem}

Then, the basic properties of the relation $\sim$ are given in the following lemma.

\begin{lem}
For any continuous map $f:X\to X$, the relation $\sim$ on $CR(f)$ is a closed $(f\times f)$-invariant equivalence relation. In other words, $\sim$ is an equivalence relation on $CR(f)$, $R=\{(x,y)\in CR(f)^2:x\sim y\}$ is a closed subset of $CR(f)^2$, and $(f\times f)(R)\subset R$.
\end{lem}

\begin{rem}
\normalfont
Actually, we have $(f\times f)(R)=R$ in the above notation.
\end{rem}

It is not difficult to give a direct proof of Lemma 3.1 based only on the definition of the relation $\sim$. However, we confirm it through an alternative description of $\sim$. In \cite{Ka1}, it was shown that for any continuous map $f:X\to X$, $CR(f)$ admits the so-called {\em $\delta$-cyclic decomposition} for each $\delta>0$, that is, a family of disjoint subsets of $CR(f)$ \[\mathcal{D}(\delta)=\{D_{i,j}:1\le i\le K,\:0\le j\le m_i-1\}\] with the following properties, where $J=\{(i,j):1\le i\le K,\:0\le j\le m_i-1\}$.

\begin{itemize}
\item[(D1)] $CR(f)=\bigsqcup _{(i,j)\in J}D_{i,j}$, and every $D_{i,j}$, $(i,j)\in J$, is clopen in $CR(f)$.
\item[(D2)] Putting $D_{i,m_i}=D_{i,0}$, we have $f(D_{i,j})=D_{i,j+1}$ for every $(i,j)\in J$.
\item[(D3)] Given any $x,y\in D_{i,j}$ with $(i,j)\in J$, there is $N>0$ such that for any integer $n\ge N$, there is a $\delta$-chain $(x_\eta)_{\eta=0}^k\subset CR(f)$ of $f$ with $x_0=x$, $x_k=y$, and $k=m_in$.
\end{itemize}

Then, Lemma 3.1 is an immediate consequence of the above properties (D1), (D2), and the following lemma.

\begin{lem}
Let $f:X\to X$ be a continuous map. Then, for any $x,y\in CR(f)$, $x\sim y$ if and only if for every $\delta>0$, $x$ and $y$ are contained in the same component of $\mathcal{D}(\delta)$.
\end{lem}

\begin{proof}
Suppose $x\sim y$. For any given $\delta>0$, let $\mathcal{D}(\delta)=\{D_{i,j}:1\le i\le K,\:0\le j\le m_i-1\}$ be the $\delta$-cyclic decomposition of $CR(f)$. Without loss of generality, we may assume $x\in D_{1,0}$. Then, it must be shown that $y\in D_{1,0}$. For the purpose, we take $\delta'>0$ such that $d(A,B)=\inf\{d(a,b):a\in A, b\in B\}>\delta'$ for any $A,B\in\mathcal{D}(\delta)$ with $A\ne B$. For any integer $n=qm_1+r\ge0$, where $q\ge0$ and $0\le r\le m_1-1$, put $D_{1,n}=D_{1,r}$. Then, for every $\delta'$-chain $(x_n)_{n=0}^L\subset CR(f)$ of $f$ with $x_0=x$, by (D2) and the choice of $\delta'$, we have $x_n\in D_{1,n}$ for all $0\le n\le L$. Now, since $x\sim y$, there are $m=m(\delta')>0$ and $N=N(\delta')>0$ as in Definition 3.1. Put $M=mm_1N$ and note that $m_1N\ge N$. Thus, there is a $\delta'$-chain $(x_n)_{n=0}^M\subset CR(f)$ of $f$ such that $x_0=x$ and $x_M=y$, and this implies $y=x_{(mN)m_1}\in D_{1,(mN)m_1}=D_{1,0}$. The converse is a direct consequence of (D3).
\end{proof}

The following lemma relates the entropy pairs with the chain relation $\sim$.

\begin{lem}
Let $f:X\to X$ be a continuous map. Then, for any $(x,y)\in E(X,f)$, $(x,y)\in CR(f)^2$ and $x\sim y$.
\end{lem}

\begin{proof}
Suppose $(x,y)\in E(X,f)$. Given any $\delta>0$, let $\mathcal{D}(\delta)$ be the $\delta$-cyclic decomposition of $CR(f)$. We define a relation $R\subset X^2$ by \[R=\{(a,a)\in X^2:a\in X\}\cup\{(a,b)\in CR(f)^2:\text{$\{a,b\}\subset A $ for some $A\in\mathcal{D}(\delta)$}\}.\] Then, by (D1) and (D2), $R$ is a closed $(f\times f)$-invariant equivalence relation. Let $X_R$ be the quotient space, $\pi:X\to X_R$ be the quotient map, and $f_R:X_R\to X_R$ be the continuous map defined by $f_R\circ\pi=\pi\circ f$. $\pi$ is a factor map. We easily see that $\Omega(f_R)=\{\pi(A):A\in\mathcal{D}(\delta)\}$, a finite set (here $\Omega(\cdot)$ is the non-wandering set), and hence $h_{top}(f_R)=h_{top}(f_R|_{\Omega(f_R)})=0$. From Lemma 2.1, it follows that $\pi(x)=\pi(y)$ (or $(x,y)\in R$), which implies $\{x,y\}\subset A$ for some $A\in\mathcal{D}(\delta)$. Since $\delta>0$ is arbitrary, by Lemma 3.2, we conclude that $(x,y)\in CR(f)^2$ and $x\sim y$.
\end{proof}

\begin{rem}
\normalfont
For any $(x,y)\in CR(f)^2$, if $x\sim y$, then we easily see that the following holds: Given any $\delta>0$, there are an integer $a>0$ and $\delta$-chains $\gamma_{ij}=(y_{ij,\eta})_{\eta=0}^a$ of $f$ with $y_{ij,0}=i$ and $y_{ij,a}=j$ for all $i,j\in\{x,y\}$ (P). Conversely, it is obvious that if (P) holds for $(x,y)\in CR(f)^2$, then $x\sim y$. Keeping this equivalence in mind, Lemma 3.3 can also be proved as below by Kerr and Li characterization of the entropy pairs as the so-called (non-diagonal) {\em IE-pairs} (see \cite{KL} for details).
\end{rem}

\begin{proof}
Let $(x,y)\in E(X,f)$. Then, according to \cite{KL}, $(x,y)$ is an IE-pair for $f$. It follows that for any $\epsilon>0$, $(B_\epsilon(x), B_\epsilon(y))$ has an {\em independence set} of positive density (here $B_\epsilon(\cdot)$ is the $\epsilon$-ball). In particular, this implies the existence of two integers $0\le m<n$ such that for all $i,j\in\{x,y\}$, there is $x_{ij}\in X$ with $f^m(x_{ij})\in B_\epsilon(i)$ and $f^n(x_{ij})\in B_\epsilon(j)$. Then, it is clear that the property (P) holds for $(x,y)$, thus $(x,y)\in CR(f)^2$ and $x\sim y$.
\end{proof}

\section{Proof of Theorem 1.2}

In this section, we prove Theorem 1.2. For the purpose, we first define a separation property for two chain recurrent points.

\begin{defi}
Given a continuous map $f:X\to X$, we say that a pair of points $(x,y)\in CR(f)^2$ has property* if there is $r>0$ such that for any $\delta>0$, there are two $\delta$-cycles $\gamma_1=(x_i)_{i=0}^k$, $\gamma_2=(y_i)_{i=0}^k\subset CR(f)$ of $f$ such that $x_0=x_k=x$, $y_0=y_k=y$, and $d(x_i,y_i)>r$ for each $0\le i\le k$.
\end{defi}

Let $f:X\to X$ be a continuous map. A pair of points $(x,y)\in X^2$ is said to be a {\em proximal pair} for $f$ if $\liminf_{n\to\infty}d(f^n(x),f^n(y))=0$. If $(x,y)\in X^2$ is not proximal, i.e. $\inf_{n\ge0}d(f^n(x),f^n(y))>0$, the pair is said to be {\em distal}. 

\begin{rem}
\normalfont
(1) If $(x,y)\in CR(f)^2$ has property*, then $(x,y)$ is a distal pair for $f$.
$ $\newline
(2) If $(x,y)\in X^2$ is a distal pair for $f$, then any $(z,w)\in\omega((x,y),f\times f)\subset CR(f)^2$ satisfies property*.
$ $\newline
(3) Assume that  $(x,y)\in CR(f)^2$ with $x\sim y$ is a distal pair for $f$. Then, given any $(z,w)\in\omega((x,y),f\times f)\subset CR(f)^2$, since $\sim$ is a closed $(f\times f)$-invariant relation (by Lemma 3.1), $z\sim w$. It also holds that $(z,w)$ has property*.
\end{rem}

The following lemma gives a sufficient condition for DC1.

\begin{lem}
Suppose that a continuous map $f:X\to X$ satisfies the following properties.
\begin{itemize}
\item[(1)] There is $(z,w)\in CR(f)^2$ which satisfies $z\sim w$ and property*.
\item[(2)] For any sequence $(x_i)_{i\ge0}$ of points in $CR(f)$, if $\lim_{i\to\infty}d(f(x_i),x_{i+1})=0$, then there is $x\in X$ such that $\lim_{n\to\infty}\frac{1}{n}\sum_{i=0}^{n-1}d(f^i(x),x_i)=0$.
\end{itemize}
Then, $f$ exhibits distributional chaos of type 1 (DC1).
\end{lem}

\begin{proof}
By (1), there is $r>0$ with the following property: For any integer $n\ge1$, there are an integer $a_n>0$, two $n^{-1}$-cycles $\gamma_{0,n}=(p_i)_{i=0}^{a_n}$, $\gamma_{1,n}=(q_i)_{i=0}^{a_n}\subset CR(f)$ of $f$, and two $n^{-1}$-chains $\alpha_n=(r_i)_{i=0}^{a_n}$, $\beta_n=(s_i)_{i=0}^{a_n}\subset CR(f)$ of $f$ such that $p_0=p_{a_n}=z$, $q_0=q_{a_n}=w$, $r_0=s_{a_n}=z$, $r_{a_n}=s_0=w$, and $d(p_i,q_i)>r$ for every $0\le i\le a_n$.  We take a sequence of integers $0<m_1<m_2<\cdots$ such that $m_1=2$, and, putting $b_n=\sum_{i=0}^{n-1}a_im_i$, we have
\begin{equation*}
\frac{a_n(m_n-2)+1}{b_n+a_nm_n+1}>1-n^{-1}
\end{equation*}
for every $n>1$. Put $c_{0,n}=\gamma_{0,n}^{m_n}$ and $c_{1,n}=\alpha_n\gamma_{1,n}^{m_n-2}\beta_n$ for each $n\ge1$. We have $l(c_{0,n})=l(c_{1,n})=a_nm_n$ for any $n\ge1$ (here $l(\cdot)$ denotes the length of the cycle). Then, for each $u=(u_n)_{n\in\mathbb{N}}\in\{0,1\}^\mathbb{N}$, define $\xi(u)=c_{u_1,1}c_{u_2,2}c_{u_3,3}\cdots\subset CR(f)$, a limit pseudo orbit of $f$, and take $x(u)$ as in the property (2). Let us fix an uncountable subset $S\subset\{0,1\}^\mathbb{N}$ such that for any $u,v\in S$ with $u\ne v$, both $\{n\in\mathbb{N}:u_n=v_n\}$ and $\{n\in\mathbb{N}:u_n\ne v_n\}$ are infinite sets. We shall prove that for any $u,v\in S$ with $u\ne v$, $(x(u),x(v))\in X^2$ is a DC1-pair for $f$.

For the purpose, put $\xi(u)=(x_{0,i})_{i\ge0}$, $\xi(v)=(x_{1,i})_{i\ge0}$, and $(x(u),x(v))=(x_0,x_1)$. Note that, putting $\epsilon_{j,n}=\frac{1}{n}\sum_{i=0}^{n-1}d(f^i(x_j),x_{j,i})$, we have $\lim_{n\to\infty}\epsilon_{j,n}=0$ for each $j\in\{0,1\}$. Given any $\delta>0$, putting $c_n=b_n+a_nm_n+1$, we have
\begin{equation*}
|\{0\le i\le c_n-1: d(f^i(x_j),x_{j,i})<\frac{\delta}{2}\}|\ge c_n(1-2\delta^{-1}\epsilon_{j,c_n})
\end{equation*}
for any $n\ge1$ and $j\in\{0,1\}$. On the other hand, for any $n>1$, if $u_n=v_n$, then because $c_{u_n,n}=c_{v_n,n}$,
\begin{equation*}
\{b_n\le i\le b_n+a_nm_n\}\subset\{0\le i\le c_n-1:x_{0,i}=x_{1,i}\},
\end{equation*}
and so
\begin{equation*}
|\{0\le i\le c_n-1:x_{0,i}=x_{1,i}\}|\ge c_n\cdot\frac{a_nm_n+1}{c_n}>c_n(1-n^{-1}).
\end{equation*}
Hence, for any $n>1$ with $u_n=v_n$, we have
\begin{equation*}
|\{0\le i\le c_n-1:d(f^i(x_0),f^i(x_1))<\delta\}|>c_n(1-n^{-1}-2\delta^{-1}\epsilon_{0,c_n}-2\delta^{-1}\epsilon_{1,c_n}).
\end{equation*}
This together with $|\{n\in\mathbb{N}:u_n=v_n\}|=\infty$ yields
\begin{equation*}
\limsup_{n\to\infty}\frac{1}{n}|\{0\le i\le n-1:d(f^i(x_0),f^i(x_1))<\delta\}|=1.
\end{equation*}
Note that $\delta>0$ is arbitrary. It only remains to prove the following:
\begin{equation*}
\limsup_{n\to\infty}\frac{1}{n}|\{0\le i\le n-1:d(f^i(x_0),f^i(x_1))>\frac{r}{3}\}|=1.
\end{equation*}
Similarly as above, we have
\begin{equation*}
|\{0\le i\le c_n-1: d(f^i(x_j),x_{j,i})<\frac{r}{3}\}|\ge c_n(1-3r^{-1}\epsilon_{j,c_n})
\end{equation*}
for any $n\ge1$ and $j\in\{0,1\}$. On the other hand, for any $n>1$, if $u_n\ne v_n$, then because $c_{u_n,n}\ne c_{v_n,n}$,
\begin{equation*}
\{b_n+a_n\le i\le b_n+a_n+a_n(m_n-2)\}\subset\{0\le i\le c_n-1:d(x_{0,i},x_{1,i})>r\},
\end{equation*}
and so
\begin{equation*}
|\{0\le i\le c_n-1:d(x_{0,i},x_{1,i})>r\}|\ge c_n\cdot\frac{a_n(m_n-2)+1}{c_n}>c_n(1-n^{-1}).
\end{equation*}
Thus, for any $n>1$ with $u_n\ne v_n$, we have
\begin{equation*}
|\{0\le i\le c_n-1:d(f^i(x_0),f^i(x_1))>\frac{r}{3}\}|>c_n(1-n^{-1}-3r^{-1}\epsilon_{0,c_n}-3r^{-1}\epsilon_{1,c_n}).
\end{equation*}
This together with $|\{n\in\mathbb{N}:u_n\ne v_n\}|=\infty$ implies the required equation. Now, It has been proved that for any $u,v\in S$ with $u\ne v$, $(x(u),x(v))$ is a DC1-pair for $f$. Since $S$ is an uncountable set, $f$ exhibits DC1.
\end{proof}

We need the following lemma from \cite{Ka2}.

\begin{lem}\cite[Theorem 1.1]{Ka2}
Let $f:X\to X$ be a continuous map with the limit shadowing property. Then, $f|_{CR(f)}:CR(f)\to CR(f)$ satisfies the shadowing property.
\end{lem}

It is known that the shadowing property with positive topological entropy allows us to obtain a subsystem of $(X,f^a)$ (for some $a>0$) which has $(\{0,1\}^\mathbb{N},\sigma)$ (here $\sigma$ is the shift map) as a factor (see, for example, \cite{LO, MoO}). We use it and Lemma 3.3 to find a pair of chain recurrent points with the two properties assumed in Lemma 4.1. Since, especially, the second part of the proof of Theorem 1.2 relies on a specific construction of the subsystem through $\sim$, for completeness, we briefly describe it below.

\medskip
\noindent
{\em Construction}: Suppose that $g=f|_{CR(f)}:CR(f)\to CR(f)$ has the shadowing property. Given $(x,y)\in CR(f)^2$ with $x\ne y$ and $x\sim y$, put $x_0=x$, $x_1=y$, and let $0<\epsilon<2^{-1}d(x,y)$. Take $\delta=\delta(\epsilon)>0$ such that every $\delta$-pseudo orbit of $g$ is $\epsilon$-shadowed by some point of $CR(f)$. Since $x\sim y$, there are an integer $a>0$ and $\delta$-chains $\gamma_{ij}=(y_{ij,\eta})_{\eta=0}^a\subset CR(f)$ of $g$ with $y_{ij,0}=x_i$ and $y_{ij,a}=x_j$ for each $i,j\in\{0,1\}$. For any $s=(s_i)_{i\in\mathbb{N}}\in\{0,1\}^\mathbb{N}$, consider the $\delta$-pseudo orbit $\xi(s)=\gamma_{s_1s_2}\gamma_{s_2s_3}\gamma_{s_3s_4}\cdots\subset CR(f)$ of $g$, which is $\epsilon$-shadowed by some $x(s)\in CR(f)$. Let
\begin{equation*} 
Y=\{x\in CR(f): \text{$\xi(s)$ is $\epsilon$-shadowed by $x$ for some $s\in\{0,1\}^\mathbb{N}$}\}
\end{equation*}
and note that $Y$ is a compact $f^a$-invariant subset of $CR(f)$. Define a map $\pi:Y\to\{0,1\}^\mathbb{N}$ by the condition that $\xi(\pi(x))$ is $\epsilon$-shadowed by $x$ for each $x\in Y$. Then, it is easy to see that $\pi:(Y,f^a)\to(\{0,1\}^\mathbb{N},\sigma)$ is a factor map. 

\begin{rem}
\normalfont
By the above construction, we see that if $f|_{CR(f)}$ has the shadowing property, then for any $(x,y)\in CR(f)^2$ with $x\ne y$, $x\sim y$ implies $(x,y)\in E(X,f)$. Indeed, we have $x(s)\in\bigcap_{i\in\mathbb{N}}f^{-(i-1)a}(B_\epsilon(x_{s_i}))\ne\emptyset$ for every $s\in\{0,1\}^\mathbb{N}$ in the above notation, which implies \[h(f,\{B_\epsilon(x)^c,B_\epsilon(y)^c\})\ge a^{-1}\log2>0.\] Since $0<\epsilon<2^{-1}d(x,y)$ is arbitrary, $(x,y)\in E(X,f)$. The fact that the shift map $\sigma:\{0,1\}^\mathbb{N}\to\{0,1\}^\mathbb{N}$ has upe, which is mentioned in Section 2, is a consequence of this remark.
\end{rem}

Combining this remark with Lemma 3.3 (and Lemma 2.1), we get the following corollary, which characterizes the set of entropy pairs by the relation $\sim$, under the assumption of the shadowing property.

\begin{cor}
Let $f:X\to X$ be a continuous map. If $f|_{CR(f)}:CR(f)\to CR(f)$ has the shadowing property, then it holds that \[E(X,f)=\{(x,y)\in CR(f)^2:x\ne y\:\:\text{and}\:\:x\sim y\}.\] In particular, $h_{top}(f)>0$ if only if $x\sim y$ for some $(x,y)\in CR(f)^2$ with $x\ne y$.
\end{cor}

Finally, let us prove Theorem 1.2.

\begin{proof}[Proof of Theorem 1.2]
By (1) and Lemma 4.2, $f|_{CR(f)}:CR(f)\to CR(f)$ satisfies the shadowing property. Since $E(X,f)\ne\emptyset$ by (2) and Lemma 2.1, we can take $(x,y)\in E(X,f)$. Then, by Lemma 3.3, it holds that $(x,y)\in CR(f)^2$, $x\ne y$, and $x\sim y$; therefore, we have a factor map $\pi:(Y,f^a)\to(\{0,1\}^\mathbb{N},\sigma)$ as above. Since $\sigma$ has upe, especially, we have $(0^\infty,1^\infty)\in E(\{0,1\}^\mathbb{N},\sigma)$. According to Lemma 2.1, there is $(p,q)\in E(Y,f^a)$ such that $\pi(p)=0^\infty$ and $\pi(q)=1^\infty$. Because $(0^\infty,1^\infty)$ is a distal pair for $\sigma$, $(p,q)$ is a distal pair for $f^a$, so for $f$. On the other hand, since \[E(Y,f^a)\subset E(X,f^a)\subset E(X,f),\] we have $(p,q)\in E(X,f)$, which together with Lemma 3.3 implies $(p,q)\in CR(f)^2$ and $p\sim q$. Fix $(z,w)\in\omega((p,q),f\times f)\subset CR(f)^2$. Then, as Remark 4.1, it satisfies $z\sim w$ and property*. Thus, by (1) and Lemma 4.1, we conclude that $f$ exhibits DC1.

Let us prove $E(X,f)\subset\overline{DC1(X,f)}$. Given any $(x,y)\in E(X,f)$, Lemma 3.3 implies $(x,y)\in CR(f)^2$, $x\ne y$, and $x\sim y$, so we have a factor map $\pi:(Y,f^a)\to(\{0,1\}^\mathbb{N},\sigma)$ as above (for any $0<\epsilon<2^{-1}d(x,y)$). Take $(p,q)\in E(Y,f^a)$ such that $\pi(p)=0^\infty$ and $\pi(q)=1^\infty$. Fix $(z,w)\in\omega((p,q),f^a\times f^a)$ and note that
\begin{equation*}
(z,w)\in\omega(p,f^a)\times\omega(q,f^a)\subset\pi^{-1}(0^\infty)\times\pi^{-1}(1^\infty)\subset B_\epsilon(x)\times B_\epsilon(y).
\end{equation*}
Since $\omega((p,q),f^a\times f^a)\subset\omega((p,q),f\times f)$, similarly as above, $(z,w)\in CR(f)^2$ satisfies $z\sim w$ and property*. Then, we see that, because $f$ has the limit shadowing property, in the proof of Lemma 4.1, $x(u)$ can be taken as a limit shadowing point of $\xi(u)$ for any $u\in\{0,1\}^\mathbb{N}$. Also, in the proof of Lemma 4.1, it has been proved that for any $u,v\in\{0,1\}^\mathbb{N}$ with $|\{n\in\mathbb{N}:u_n=v_n\}|=\infty$ and $|\{n\in\mathbb{N}:u_n\ne v_n\}|=\infty$, $(x(u), x(v))$ is a DC1-pair for $f$. For concreteness, put $u=0^\infty$ and $v=(01)^\infty=0101\cdots$. Then, by the definition of $\xi(u)$ and $\xi(v)$, we easily see that there is a sequence of integers $0<n_1<n_2\cdots$ such that \[\lim_{j\to\infty}(f^{n_j}(x(u)),f^{n_j}(x(v)))=(z,w).\] Since every $(f^{n_j}(x(u)),f^{n_j}(x(v)))$, $j\ge1$, is a DC1-pair for $f$ as $(x(u),x(v))$ is so, we have $(z,w)\in\overline{DC1(X,f)}$. Thus, by $(z,w)\in B_\epsilon(x)\times B_\epsilon(y)$, and since $0<\epsilon<2^{-1}d(x,y)$ is arbitrary, $(x,y)\in\overline{DC1(X,f)}$, completing the proof.
\end{proof}

\appendix
\section{}

In this appendix, as a complement to the proof of Theorem 1.2, we show that the existence of $(z,w)\in CR(f)^2$ with $z\sim w$ and property* is a necessary condition for DC1. For the purpose, the next lemma is needed.

\begin{lem}
Let $f:X\to X$ be a continuous map. If $(x,y)\in X^2$ is a proximal pair for $f$, then any $(z,w)\in\omega((x,y),f\times f)\subset CR(f)^2$ satisfies $z\sim w$.
\end{lem}

\begin{proof}
Given any $\delta>0$, consider the $\delta$-cyclic decomposition $\mathcal{D}(\delta)$ of $CR(f)$. For any $p\in X$, note that $\lim_{n\to\infty}d(f^n(p),CR(f))=0$, and so by (D2), we have $\lim_{n\to\infty}$ $d(f^n(p),f^n(A_p))=0$ for some $A_p\in\mathcal{D}(\delta)$. It holds that \[\lim_{n\to\infty} d(f^n(x),f^n(A_x))=0\quad\text{and}\quad\lim_{n\to\infty} d(f^n(y),f^n(A_y))=0.\] Since $(x,y)$ is proximal, again by (D2), $A_x$ and $A_y$ should be equal. Put $A=A_x=A_y$. Then, because $(z,w)\in\omega((x,y),f\times f)$, there are $B\in\mathcal{D}(\delta)$ and a sequence of integers $0<n_1<n_2<\cdots$ such that $f^{n_i}(A)=B$ for all $i\ge1$, and $\lim_{i\to\infty}(f^{n_i}(x),f^{n_i}(y))=(z,w)$.  These properties give $\{z,w\}\subset B$, and since $\delta>0$ is arbitrary, by Lemma 3.2, we conclude $z\sim w$.
\end{proof}

\begin{prop}
Let $f:X\to X$ be a continuous map. If $(x,y)\in DC1(X,f)$, then there exists $(z,w)\in\omega((x,y),f\times f)\subset CR(f)^2$ with $z\sim w$ and property*.
\end{prop}

\begin{proof}
Let $(x,y)\in DC1(X,f)$. Then, $(x,y)$ is clearly a proximal pair for $f$. From the definition of DC1-pairs, it also follows that there is $\delta_0>0$ for which \[T=\{n\ge0:d(f^n(x),f^n(y))>\delta_0\}\] is a so-called {\em thick set}, meaning that for any $n>0$, there is $j\ge0$ such that $\{j,j+1,\dots, j+n\}\subset T$. Let us choose two sequences of integers $0<n_1<n_2<\cdots$, $0<N_1<N_2<\cdots$, and $(p,q)\in X^2$ such that the following holds.
\vspace{0.3mm}
\begin{itemize}
\item[(1)] For any $j\ge1$ and any $n_j\le n\le n_j+N_j$, $d(f^n(x),f^n(y))>\delta_0$.
\vspace{0.5mm}
\item[(2)] $\lim_{j\to\infty}(f^{n_j}(x),f^{n_j}(y))=(p,q)$.
\end{itemize}
\vspace{0.3mm}
Then, $d(f^n(p),f^n(q))\ge\delta_0$ for all $n\ge0$; therefore, $(p,q)\in\omega((x,y),f\times f)$ is distal. By Lemma A.1, it also holds that $p\sim q$. Take $(z,w)\in\omega((p,q),f\times f)\subset\omega((x,y),f\times f)$. Then, as Remark 4.1, $z\sim w$ and property* hold.
\end{proof}

We end with a simple application of Lemma A.1. A pair of points $(x,y)\in X^2$ is said to be a {\em Li-Yorke pair} for $f$ if \[\liminf_{n\to\infty}d(f^n(x),f^n(y))=0\quad\text{and}\quad\limsup_{n\to\infty}d(f^n(x),f^n(y))>0.\] By Lemma A.1, a similar argument as the proof of Proposition A.1 shows that for any Li-Yorke pair $(x,y)$ for $f$, there is $(z,w)\in\omega((x,y),f\times f)\subset CR(f)^2$ such that $z\ne w$ and $z\sim w$. When $f|_{CR(f)}:CR(f)\to CR(f)$ has the shadowing property, due to Corollary 4.1, the existence of such a $(z,w)$ implies $h_{top}(f)>0$. Thus, by Lemma 4.2, the assumption (2) of Theorem 1.2 holds if $f$ has the limit shadowing property and has a Li-Yorke pair. Also, it follows that a continuous map $f$ with the limit shadowing property exhibits DC1 only if $h_{top}(f)>0$ because any DC1-pair is a Li-Yorke pair.

\end{document}